\documentclass[12pt]{iopart}
\usepackage{iopams,amsthm} 
\usepackage{mathrsfs}

\usepackage{amsopn}

\newcommand{\proba}{\operatorname{\mathscr{P}}}

\newcommand{\diam}{\operatorname{diam}}
\newcommand{\sys}{\operatorname{sys}}
\newcommand{\multi}{\operatorname{\mathscr{M}}}
\newcommand{\Lip}{\operatorname{Lip}}
\newcommand{\Lipnorm}[1]{\left\| #1\right\|_{\mathrm{Lip}}}

\newcommand{\dd}{\mathrm{d}}

\newtheorem{theo}{Theorem}[section]
\newtheorem{prop}[theo]{Proposition}
\newtheorem{defi}[theo]{Definition}
\newtheorem{lemm}[theo]{Lemma}
\newtheorem{exem}[theo]{Example}
\newtheorem{coro}[theo]{Corollary}

\begin{document}

\title[Contraction in the Wasserstein metric]{Contraction in
  the Wasserstein metric for some Markov chains,
  and applications to the dynamics of expanding maps.}
\author{Beno\^{\i}t R. Kloeckner$^1$, Artur O. Lopes$^2$ and Manuel Stadlbauer$^3$}

\address{$^1$ Universit\'e Paris-Est, Laboratoire d'Analyse et de Mat\'ematiques Appliqu\'ees (UMR 8050), UPEM, UPEC, CNRS, F-94010, Cr\'eteil, France} 
\address{$^2$ Instituto de Matem\'atica,
Universidade Federal do Rio Grande do Sul,
Av. Bento Gon\c{c}alves, 9500,
91509-900 Porto Alegre, RS - Brazil}
\address{$^3$ Departamento de Matem\'atica, Universidade Federal do Rio de Janeiro, Ilha do Fund\~ao. 21941-909 Rio de Janeiro. RJ - Brazil. } 

\ead{\mailto{benoit.kloeckner@u-pec.fr}, \mailto{manuel@im.ufrj.br} and \mailto{arturoscar.lopes@gmail.com}}

\begin{abstract} We employ techniques from optimal transport in order to prove decay of transfer operators associated to iterated functions systems and expanding maps, giving rise to a new proof without requiring a Doeblin-Fortet (or Lasota-Yorke) inequality.

Our main result is the following. Suppose $T$ is an expanding transformation acting on a compact metric space $M$ and $A: M \to \mathbb{R}$ a given fixed H\"older function, and
denote by $\mathscr{L}$ the Ruelle operator  associated to $A$.
We show that if $\mathscr{L}$ is normalized (i.e. if $\mathscr{L}(1)=1$), then
the dual transfer operator $\mathscr{L}^*$
is an exponential contraction on the set of probability measures on $M$ with the $1$-Wasserstein
metric.

Our approach is flexible and extends to a relatively general setting, which
we name Iterated Contraction Systems. We also derive from our main result several
dynamical consequences; for example we show that Gibbs measures
depends in a Lipschitz-continuous way on variations of the potential.
\end{abstract}


\vspace{2pc}

\section{Introduction and statement of the main results}

It has already been noticed that the $1$-Wasserstein distance
issued from optimal transportation theory is very convenient to prove
exponential contraction properties for Markov chains
(see e.g. \cite{HairerMattingly:2008,Stadlbauer:2013b,Olli}).
In this article, we observe that this idea
applies very effectively to the dynamics of expanding maps: indeed the
dual transfer operator of an expanding map with respect to a
normalized potential can be
seen as a Markov chain, for which we prove exponential contraction.
We shall notably deduce from this result
several Lipschitz stability results for expanding maps: stability of
Gibbs measures
in terms of a variation of the potential, stability of the maximal
entropy measure in terms of a variation of the map, etc.

By these results and the simplicity of the proofs,
we hope that the present article will make a clear case about
the usefulness of the application of coupling techniques and objects from optimal
transport to dynamical systems and  thermodynamical formalism (general
references for this last topic are \cite{PP} and \cite{Bal}). 

Note that a similar coupling has been used in e.g. \cite{BressaudFernandezGalves:1999}
in order to show decay of correlations for Gibbs measures of
low-regularity potential in the case of the full shift.
However, in contrast to the ideas from optimal transport used in here,  
their argument is based on an estimate through a dominating Markov chain.
  
While we stick here to the more standard case of H\"older potentials, 
we take a more geometric point of view that allows us firstly to handle a much broader family of dynamical
systems and secondly to derive a number of corollaries. Namely, the contraction
in the Wasserstein metric easily implies a spectral gap and decay
of correlations, but also the stability results alluded to above.

Our main result and method of proof are also similar to a recent result
of the third named author for some random Markov shifts
(\cite{Stadlbauer:2013b}); again the present result is less general
in some aspects and more general in others since here we only consider
non-random dynamical systems but are able to cover
a wide range of expanding maps and iterated function systems.\\

We consider the following setting: let $(\Omega,d)$ be a compact metric space, $k \in \mathbb{N}$ and $F$ a map which assigns to $x \in \Omega$ a $k$-multiset $F(x) \subset \Omega$. That is, allowing multiple occurrences of elements, $F(x)$ contains $k$ elements (a typical example is given by $F(x)=T^{-1}(x)$ where $T$ is a $k$-to-$1$ map).
We then refer to $F$ as a $k$-\emph{iterated contraction system (ICS)} if there exists  $\theta <1$ such that for all $x,y \in \Omega$ there exists a bijection $x_i \mapsto y_i$ between $F(x)$ and $F(y)$ with $d(x_i,y_i)\leq \theta d(x,y)$ for all $i=1, \ldots ,k$.
We will say that a transformation $T$ of $\Omega$ is a
\emph{regular expanding} map if $T^{-1}(\{x\})$ defines an ICS
once its elements are given suitable multiplicities.
For more details we refer to section \ref{s:Definitions and examples}.

Observe that this class of dynamical systems contains, among others,
expanding local diffeomorphisms of compact Riemannian manifolds
and iterated function systems (IFS) given by $k$ contractions on $\Omega$.
A general reference for IFS is \cite{U}.

The transfer operator with respect to a given continuous function
$A: \Omega \to \mathbb{R}$ is
defined as usual by, for $f : \Omega \to \mathbb{R}$ continuous,
\[ \mathscr{L}(f)(x) = \sum_{y \in F(x)} e^{A(y)} f(y).\]
Furthermore, let $\rho$ refer to the spectral radius of $\mathscr{L}$
acting on continuous functions and suppose that
$h : \Omega \to \mathbb{R}$ is strictly positive and Lipschitz
continuous with $\mathscr{L}(h)=\rho h$; we will show that
such an $h$ exists and is
unique up to multiplication by constants in proposition
\ref{p:existence_of_eigenfunctions} and corollary \ref{c:spectralgap}
below. Then the \emph{normalized operator} defined by
\[ \mathbb{P}(x) = \mathscr{L}(h \cdot f)(x)/\rho h(x)\]
satisfies $\mathbb{P}(1)=1$ and is conjugate to $\mathscr{L}$
up to the constant $\rho$;
the iterates  are related through $\rho^n h \cdot \mathbb{P}^n(f)
= \mathscr{L}^n(h \cdot f)(x)$. By uniqueness of $h$,
$\mathbb{P}(x)$ is uniquely determined by $F$ and $A$. Also note that in
case of an ICS which is defined through a map $T$, the above operator can
be obtained by substituting $A$ by the \emph{normalized potential}
$A + \log h - \log h \circ T - \log \rho$.

Let us briefly introduce the definition of the $1$-Wasserstein metric
(the only one that we will use here) and recall some of its basic
properties.

Let $\Omega$ be a compact metric space.
The $1$-Wasserstein distance is defined on the set $\proba(\Omega)$ of (Borel)
probability measures on $\Omega$ by
\[W_1(\mu,\nu) = \inf_{\pi\in\Gamma(\mu,\nu)} \int_{\Omega\times \Omega} d(x,y) \,\dd\pi(x,y)\]
where $\Gamma(\mu,\nu)$ is the set of measures on $\Omega\times \Omega$ whose marginals
are $\mu$ and $\nu$. Elements of $\Gamma(\mu,\nu)$ are called \emph{transport plans} from
$\mu$ to $\nu$ or \emph{couplings}.

Let us quote a few basic properties: $W_1$ is indeed a metric; the infimum in its definition
is always attained by some transport plan, then called optimal and generally not unique;
the topology induced by $W_1$ is the weak-$\ast$ topology (this is only true because $\Omega$
is compact). Last, realizing the infimum in the definition of $W_1(\mu,
\nu)$ is a infinite-dimensional linear program and thus has a duality.
In this specific case, this is known as Kantorovich duality and reads:
\[W_1(\mu,\nu) = \sup_{\varphi} \Big|\int \varphi\,\dd\mu - \int\varphi\,\dd\nu\Big|\]
where the supremum is on all $1$-Lipschitz functions $\varphi:\Omega\to\mathbb{R}$.

Whenever it is needed, we will write $W_1^{d}$ to stress the underlying metric $d$;
when no confusion is expected, we will simply use the same decoration on the distance and
the Wasserstein distance (e.g. $W_1'$ will denote the Wasserstein distances with
respect to a metric $d'$). Note that
the definition of $W_1$ extends to all pair of positive measures having the same
total mass.

General references on Transport Theory and the Wasserstein distance are
\cite{Vi1}, \cite{Vi2}, \cite{AGS} and \cite{Gigli:book}.\\

Our central result is the following.
\begin{theo}[Contraction property]\label{ti:contraction}
Let $F$ be an iterated contraction system with contraction
ratio $\theta\in(0,1)$ and let
$A$ be a Lipschitz-continuous potential on $\Omega$. Then the dual
$\mathbb{P}^*$ of the normalized transfer operator $\mathbb{P}$ is
exponentially contracting on probability measures in the Wasserstein
metric. That is, for all $n\in\mathbb{N}$ and all
$\mu,\nu\in\proba(\Omega)$ we have
\[W_1((\mathbb{P}^*)^n\mu,(\mathbb{P}^*)^n\nu)
    \le C\lambda^n W_1(\mu,\nu).\]
where $C$ and $\lambda<1$ are constants depending only on
$\theta$, the Lipschitz constant $\Lip(A)$ and $\diam\Omega$.
\end{theo}

There are several features of this result that we wish to stress before giving applications. First, there is no dimension restriction: our purely metric arguments are very flexible and do not depend on a
Doeblin-Fortet inequality (also known as Ionescu-Tulcea-Marinescu or Lasota-Yorke inequality, \cite{DoeblinFortet:1937}), so that the proof also applies to, say, expanding circle maps and expanding maps on higher-dimensional manifolds.

This metric setting also enables us to extend the result
from Lipschitz
to H\"older regularity without difficulty: the result applies equally
well to $\Omega$
endowed with the metric $d^\alpha$ when $\alpha\in(0,1]$, and any potential
which is $\alpha$-H\"older in the metric $d$. The conclusion
then involves the $1$-Wasserstein metric $W_\alpha$ of $d^\alpha$
(also known as the $\alpha$-Wasserstein metric of $d$), but if needed
one can use the obvious inequalities
\[W_1 \le (\diam\Omega)^{1-\alpha}\, W_\alpha
    \le (\diam\Omega)^{1-\alpha} \, W_1^\alpha.\]
We only state our results with respect to Lipschitz regularity
to avoid making the notation heavier.

Note that the constants $C$ and $\lambda$
are explicit, though convoluted (and $\lambda$ may be much closer
to $1$ than $\theta$).

The Wasserstein metric is in our opinion a natural metric (for example 
it metrizes the weak-$\ast$
topology on probability measures when $\Omega$ is compact), but its 
relevance is much deeper, as it strongly relates to the geometry
of the phase space. One notable feature is that through Kantorovich 
duality, a control
on the $1$-Wasserstein metric implies a control on the integral
of Lipschitz functions; we will use this to provide 
below several corollaries whose proofs rely on the metric being 
$W_1$, but whose statement are free from any reference to optimal 
transport.\\

Let us now give some consequences of Theorem \ref{ti:contraction}.
Unless stated otherwise, we always consider an iterated contraction
system $F$ with contraction ratio $\theta\in(0,1)$ on a phase space
$\Omega$ and a Lipschitz potential $A$, we denote by $\mathscr{L}$
the transfer operator and by
$\mathbb{P}$ its normalization.
The dependency of constants on $\Lip(A)$, $\theta$, $\diam\Omega$
will be kept implicit and $C,\lambda$ will always denote the constants
given in Theorem \ref{ti:contraction}.

The first obvious consequence of the contraction is that
$\mathbb{P}^\ast$ fixes a unique probability measure $\mu_A$;
note that in case $F$ is given by an expanding map $T$, this
$\mu_A$ is the well-known invariant Gibbs measure associated with
the potential $A$.

We proceed with a property of classical flavor.
\begin{coro}[Spectral gap]\label{ci:spectralgap}
The action on Lipschitz functions of $\mathbb{P}$ is exponentially
contracting on a complement of the set of constant functions
(which by normalization is the $1$-eigenspace of $\mathbb{P}$).

More precisely, for each Lipschitz function
$\zeta:\Omega\to\mathbb{R}$ with $\int \zeta \,\dd\mu_A=0$,
we have
\[\Lipnorm{\mathbb{P}^n \zeta} \le C_2(\zeta) \lambda^n\]
where $C_2(\zeta)=C(1+\diam\Omega)\,\Lip(\zeta)$ and
$\Lipnorm{\cdot}=\left\| \cdot \right\|_\infty + \Lip(\cdot)$ denotes
the Lipschitz norm.
\end{coro}
This result is well-known in many cases, and the references are too
numerous to be given here; see for example the already-cited \cite{PP}.
Our method has two strengths: we obtain the result in the broad
framework of ICS, and we get explicit dependency of the constant in
term of metric quantities (diameter, Lipschitz constant, etc.)

We now turn to stability results (see Section \ref{s:Stability of the
Gibbs map}).

\begin{coro}[Lipschitz-continuity of the Gibbs map]\label{ci:Gibbs-map}
Assume that $A,B$ are normalized Lipschitz potentials for the same ICS $F$
and let $\mu_A$ and $\mu_B$
refer to the corresponding Gibbs measures. Then
\[W_1(\mu_A,\mu_B) \le C_3 \left\|  A-B\right\|_\infty\]
where $C_3=\frac{C}{1-\lambda} \diam\Omega$.\footnote{Only $\Lip(A)$
appears in $C$ and $\lambda$, by no accident: we only need to control
one Lipschitz constant, not both.}
In particular, for any Lipschitz test function $\varphi$, we have
\[\big\| \int \varphi \,\dd\mu_A - \int\varphi \,\dd\mu_B \big\|
  \le  C_3\Lip(\varphi) \,\left\|  A-B\right\|_\infty.\]
%
\end{coro}
This result is new, as far as we know. In many cases, classical
differentiability
results for the map $A\mapsto \int\varphi\,\dd\mu_A$ imply 
that it is locally Lipschitz in the Lipschitz norm, but we are not 
aware of a global result with a bound depending only on 
$\| A-B\|_\infty$ and $\Lip(A)$.

Note that if we translate Corollary \ref{ci:Gibbs-map} in $\alpha$-H\"older potentials, the Gibbs map is still locally \emph{Lipschitz}
on the space of $\alpha$-H\"older potentials, with
the space of measures endowed with $W_\alpha$.
The estimate with test functions then stands for $\alpha$-H\"older test functions.\\

We turn to results which are specific to the case of regular expanding maps; i.e.
we now assume that $F$ is obtained from a map $T$.
First, Corollary \ref{ci:Gibbs-map} implies the following.
\begin{coro}[Continuity of the metric entropy]\label{ci:metric-entropy}%
If $A$ and $B$ are normalized Lipschitz potentials, then
\[\| h(\mu_A)-h(\mu_B)\| \le C_4 \| A-B\|_\infty\]
where $C_4=\frac{C\Lip(A)}{1-\lambda} \diam\Omega+1$ and $h$ denotes the metric entropy.
\end{coro}
Continuity of the metric entropy is known in many cases, but
we obtain it at once for a wide class of expanding maps and with an 
explicit bound.\\

We are also able to deal with variations of the map $T$;
as an illustration of our method, we concentrate on a simple
case where potential variation will not interfere.
We will use the following notation
for the uniform distance between maps acting on the same space:
\[d_\infty(T_1,T_2) := \sup_{x\in\Omega} d(T_1(x),T_2(x)).\]

In the next result $\sys(\Omega)$ denotes the systole of the manifold $\Omega$,
i.e. the length of the shortest non-homotopically trivial curve (see \cite{Gro} for general results and references on the topic).

\begin{coro}[Continuity of the maximal entropy measure]\label{ci:max-entropy}
Let $T_1$ and $T_2$ be two $C^1$ expanding maps on the same manifold $\Omega$ with the same number $k$ of sheets, assume that
one of them is $1/\theta$-expanding, and
let $\mu_i$ be the maximal entropy measure of $T_i$ for $i=1,2$.

 If
$\| T_1-T_2\|_\infty\le \frac14  \sys(\Omega)$
then
\[W_1(\mu_1,\mu_2) \le C_5 \, d_\infty(T_1,T_2)\]
where $C_5=\frac{2C}{1-\lambda}$ and $C$ is computed with
$\Lip(A)=0$.
\end{coro}
The continuity of the maximal entropy measure is known in some cases,
see notably the work of Raith \cite{Raith1}, \cite{Raith2}. Again,
our result benefits from precise estimates and broad generality 
(although we do not cover all cases covered by the above references).

The restriction on $\left\| T_1-T_2\right\|_\infty$ can possibly
be waived; e.g. it would be sufficient to prove
that the space of expanding maps on a manifold is connected by small
jumps.

It is also very likely that Corollary \ref{ci:max-entropy}
extends in some form to many other classes of expanding maps (e.g. piecewise uniformly expanding interval maps),
but we do not have a general argument that would avoid a cumbersome list of specific results;
its main part is a general result, Corollary \ref{c:Gibbs-map} below.

Note that Corollary \ref{ci:max-entropy} deals with the regularity of a natural invariant measure
in terms of a varying expanding map, in the same spirit of many previous works (see \cite{Ru}, \cite{BS}, \cite{Ba}, \cite{Ha} and \cite{BCV})
in which the absolutely invariant measure was considered. These papers are all in the so called Linear Response Theory. Here, the maximal entropy measures we deal
with are most of the time singular with respect to Lebesgue measure and singular one with respect to
the other, a setting where many previous approaches are difficult to apply.\\

Our method depends on an argument which only applies to operators $\mathscr{L}^*$ when they map probability measures to probability measures. Therefore, it is essential to normalize these operators,
thus to have a Ruelle-Perron-Frobenius theorem in the setting
of ICS. This is the role of Proposition
\ref{p:existence_of_eigenfunctions}, and it is worth noting that the
method of proof, even though obviously inspired by the construction of
conformal measures in \cite{DenkerUrbanski:1991b}, seems to be new.
We require in corollaries
\ref{ci:metric-entropy} and \ref{ci:max-entropy}
that the potentials  are already normalized,
as we would otherwise need to control the variations of the map
that sends a potential to its normalized counterpart. While this map
is probably known to be locally Lipschitz for quite some time, it
is difficult to locate such a result in the classical literature;
in \cite{GKLM} a proof of this fact is given, which could be made
effective (i.e. giving an explicit local Lipschitz constant in term
of the potential). It follows that the metric entropy and the maximal entropy measure are locally Lipschitz-continuous (in the potential and the expanding map respectively) even without the normalization condition. The constants $C_4$ and $C_5$ should then be adjusted, but could certainly be made explicit.\\

Note that below, we introduce a pretty general framework which enables us
to treat IFS in the same
setting as expanding maps; our main motivation for this is simply to treat
expanding maps
on manifolds and piecewise uniformly expanding (onto) maps together; but
an IFS comes naturally with a
transfer operator, to which most of the above results apply. In
particular, it is possible to deduce
from our results that two self similar IFS which are close one to another
have their ``natural measures''
close one to the other.

\section{Definitions and examples}\label{s:Definitions and examples}

In this section we introduce the precise  setting in which we will work.
We tried to set unified notation applicable in as broad a generality as possible,
which explains why our definitions are not totally standard.

\subsection{Iterated contraction systems}

Iterated contraction systems, to be defined below, are a natural generalization
of iterated function systems. The only departure from the
usual setting is that instead of considering a finite set of contracting maps,
we consider one multiset-valued map with contraction properties. The reason for this
choice is that it makes this notation immediately applicable to expanding maps, see Section
\ref{s:expanding}

\begin{defi}
We shall define a \emph{multiset} with $k$ elements (or $k$-multiset)
as the orbit of a $k$-tuple
under the action of the permutation group $S_k$; we will denote a multiset
using the usual set braces, repeating elements if needed: for example
$\{1,2,2,5\}$ is a multiset with $4$ elements.

The set of elements of a multiset is called its \emph{underlying set}.

Then the \emph{multiplicity function} $\mathbf{1}_A$ of a multiset $A$ whose elements
are in some ``universal'' set $\Omega$ is the functions which maps every element
of $\Omega$ to its multiplicity as an element of $A$; the multiplicity function
contains all the information on $A$. The \emph{sum} of multisets
$A$ and $B$ is the multiset $A\uplus B$
whose multiplicity function is $\mathbf{1}_A+\mathbf{1}_B$.

A \emph{bijection} $f$
between $k$-multisets $A$ and $B$ is the data of $k$ pairs $(a_i,b_i)$
such that $A=\{a_1,\dots, a_k\}$ and $B=\{b_1,\dots,b_k\}$; beware that the
functional notation $b_i=f(a_i)$ would be misleading as we could have
$a_i=a_j$ while $f(a_i)\neq f(a_j)$; we therefore sometimes write
$f(i)=(a_i,b_i)$, with the understanding that for any permutation $\pi$,
the map $f_\pi:=i\mapsto(a_{\pi(i)},b_{\pi(i)})$ is identified with $f$.

The set of all $k$-multisets whose elements are taken in some set $\Omega$ is denoted
by $\multi_k(\Omega)$.

When summing and multiplying over multisets,
each element appears in the sum as many times as it appears in the multiset:
\[\sum_{x\in\{1,2,2,5\}} x = 1+2+2+5.\]
\end{defi}

Let us give a few motivating examples.
\begin{exem}
Consider an IFS, that is a family of $k$ contracting
maps $F_1,\dots, F_k$ of $\Omega$. The multiset valued map defined by
$F(x)=\{F_1(x),\dots, F_k(x)\}$ is an ICS: the bijection between $F(x)$
and $F(y)$ is simply given by the pairs $(F_i(x),F_i(y))$. The contraction
ratio of $F$ is the largest contraction ratio of the $F_i$.

This is a very particular kind of ICS, since we have globally defined
sections of $F$ (i.e., maps that selects continuously for each $x$ an
element of $F(x)$); but $F(x)$ is not a set whenever two $F_i$'s take the same value
at $x$.
\end{exem}

\begin{defi}
Let $\Omega$ be a complete metric space, $k$ be a positive integer,
and $F$ be a map $\Omega\to \multi_k(\Omega)$.

We say that $F$ is an \emph{iterated contraction system}
(\emph{ICS} for short, $k$-ICS or ICS with $k$ terms if we want to make $k$ explicit)
if  there is a number $\theta\in(0,1)$ (called  \emph{contraction ratio})
such that for all $x,y\in \Omega$ there is a bijection $f=(x_i,y_i)_i$
between $F(x)$ and $F(y)$ such that for all $i$,
\[d(x_i,y_i)\le \theta d(x,y).\]
The iterates of $F$ are the ICS $F^t:\Omega\to \multi_{k^t}(\Omega)$
(where $t\in \mathbb{N}$) defined by
\[F^1=F \quad\mbox{and}\quad F^{n+1}(x)= \biguplus_{y\in F^n(x)} F(y);\]
note that $\theta^n$ is a contraction ratio for $F^n$.

If $A$ is a subset of $\Omega$, we denote by $F(A)$ the union of all the
underlying sets of the $F(a)$, when $a$ runs over $A$.
\end{defi}

\begin{exem}
Consider the map
\[T:x\mapsto 2x \;\mathrm{ mod }\; 1\]
acting on $S^1=\mathbb{R}/\mathbb{Z}$,
and for each $x\in S^1$ let $F(x)=T^{-1}(\{x\})$.  Then $F$ is an ICS with contraction
ratio $1/2$.

This is a very particular kind of ICS, since $F(x)$ is always a set; but as is well-known
we do not have globally defined sections, so that it is not possible to obtain
$F$ from an IFS. However, this ICS has the nice property that each $x$ admits a neighborhood
on which sections can be defined (we say that $F$ admits local sections).
\end{exem}

\begin{exem}
The following map acting on the closed unit disc of $\mathbb{C}$ is an ICS
with contraction ratio $1/2$:
\[F : r e^{2i\pi\alpha} \mapsto \Big\{\frac{r}2 e^{i\pi\alpha},\frac{r}2 e^{i\pi(\alpha+1)} \Big\}\]
Note that $F(x)$ is a set except when $x=0$, as $F(0)=\{0,0\}$. This ICS does not even
admit \emph{local} sections around the origin.
\end{exem}

Just like an IFS, an ICS admits a unique attractor, i.e.
a non-empty compact set $A$ such that $A=F(A)$
(proof: the map $A\mapsto F(A)$ is a contraction in the Hausdorff metric,
thus has a unique fixed point). Moreover this attractor can be approximated by iterating
$F$ on any given non-empty compact set.

\subsection{Markov chains associated to an ICS and potentials}

Let $F$ be an ICS on a complete metric space $\Omega$; up to restricting
$F$ to its attractor, we assume that $\Omega$ is compact and that $\Omega=F(\Omega)$.

\begin{defi}
A Markov chain on $\Omega$ is said to be \emph{compatible}
with $F$ if at each $x\in \Omega$, its kernel $P(x,\cdot)$ is supported
on the underlying set of $F(x)$. In other words, if the position at time $t$
of the Markov chain is $x$, we ask that with probability one the position
at time $t+1$ is an element of $F(x)$.
\end{defi}

Note that compatibility only depends on the underlying set-valued map of $F$.
We will be interested by very specific compatible Markov chains, where the
transition probabilities are given by a normalization of a potential function only depending on the target points: these Markov chains
indeed occur in the thermodynamical formalism, which is our main motivation.

\begin{defi}
A \emph{potential} is simply a continuous function $A:\Omega\to \mathbb{R}$;
it is said to be \emph{normalized} with respect to $F$ if for all $x\in \Omega$ we have
\[\sum_{y\in F(x)} e^{A(y)} = 1,\]
where we sum over the \emph{multi}set $F(x)$.

The \emph{Markov chain associated to} a normalized potential $A$ is defined by letting
$m \cdot e^{A(y)}$ be the transition probability from $x$ to $y$ whenever
$y$ is an element of $F(x)$ of multiplicity $m$.

We denote by $\mathscr{L}^*_{F,A}$ (leaving asside any subscripts that are
clear from the context) the operator on finite, signed
measures, defined by
\[\int \varphi(x) \,\dd(\mathscr{L}^*\mu)(x) = \int \sum_{y\in F(x)} e^{A(y)}\varphi(y) \,\dd\mu(x)\]
whenever $\varphi$ is a continuous test function. In other words,
$\mathscr{L}^*$ is the dual of the transfer operator defined by
\[\mathscr{L}\varphi(x) = \sum_{y\in F(x)} e^{A(y)}\varphi(y).\]
Note that, if $A$ is normalized, then $\mathscr{L}(1) =1$ and $\mathscr{L}^*$ maps probability measures to probability measures.
\end{defi}

In case of a non-normalized potential, the associated Markov chain is obtained through a normalization of
$\mathscr{L}$ through the construction of an invariant function in proposition \ref{p:existence_of_eigenfunctions} as shown below (see definition \ref{d:Markov_chain_non_normalized}).

The simplest example of a normalized potential is the constant one:
$A(y)=-\log k$ where $k$ is the number of terms of $F$. For example
if $F$ is an IFS with uniform contraction ratio, the stationary
probability of the Markov chain associated to $A$ is the usual
canonical measure on the fractal attractor defined by $F$.

Other examples are easy to construct when $F$ is an IFS with the
``strong separation property'': $F$ has global sections $F_1,\dots, F_k$
with disjoint images, and any sufficiently negative continuous function
on $F_1(\Omega)\cup\dots\cup F_{k-1}(\Omega)$ can be extended to a normalized
potential by suitably choosing its values on $F_k(\Omega)$.

\subsection{The case of expanding maps}\label{s:expanding}
The definition of expanding maps may vary in the literature; the one
we adopt fits what we will need in the proof of the contraction property,
and includes in the same framework shifts, some IFS, classical smooth
expanding maps, piecewise expanding unimodal maps and other examples.

\begin{defi}
If $\Omega$ is a compact metric space, a continuous map $T:\Omega\to \Omega$ is said to be
\emph{regular expanding} if $T^{-1}:x\mapsto T^{-1}(\{x\})$ is the underlying set-valued map
of a $k$-ICS $F$, where $k=\max\{\#T^{-1}(\{x\}) \mid x\in \Omega\}$.

We say that $T$ has $k$ sheets, and if $\theta$ is a contraction ratio of $F$
then we say that $T$ is $\frac1\theta$-expanding.
\end{defi}

It is not clear from this definition that $F$ is uniquely defined by $T$; but in the
cases we will consider, the set of points $x$ having a maximal number of inverse images
is dense in $\Omega$, so that $F$ is in fact uniquely defined by $T$.

\begin{exem}
Let $\Omega$ be a compact Riemannian manifold, and $T:\Omega\to \Omega$ be a $C^1$ map such that
$\left\| D_xT(v) \right\|\ge \frac1\theta \left\| v \right\|$ for some $\theta\in(0,1)$
and all $(x,v)\in T \Omega $. Then $T$ is regular expanding; indeed
$T$ is a local diffeomorphism, thus a covering map and $F(x)=T^{-1}(x)$ defines an IFS:
the uniformly
expanding property of $D_xT$ easily ensures the contracting property for $F$,
using the lifting property on a minimizing geodesic from $x$ to $y$ to pair their
inverse images.

Note that few manifolds admit expanding maps, an obvious example being the torus
of any dimension. The keyword here is ``infra-nil-manifold'', but we will
not elaborate on this topic.
\end{exem}

\begin{exem}
Let $\Omega=[a,b]$ be a closed interval, and $T:\Omega\to\Omega$
be a piecewise $C^1$ expanding unimodal map;
that is, for some $c\in(a,b)$ the map $T$ is $C^1$ with $T'>1$ on $[a,c]$ and $C^1$
with $T'<-1$ on $[c,b]$, and we have $T(a)=T(b)=a$ and $T(c)=b$.

Then $T$ is regular expanding; it has $2$ sheets and is $(\min|T'|)^{-1}$-expanding,
and its associated ICS $F$ is in fact an IFS (the linear order on $[a,b]$ enables one
to define global sections). For all $x\neq b$, $F(x)$ has two distinct elements
while $F(b)=\{c,c\}$.
\end{exem}

More examples of this kind are provided by letting $T(x)$ zig-zag between $a$
and $b$ more than once, or by considering higher-dimensional analogues, such as the following
triangle foldings.

\begin{exem}
Let $\Omega$ be a simplex in $\mathbb{R}^d$ which is subdivided into a tiling of smaller simplices.
Consider a map
$\varphi$ defined on the vertices of this simplicial decomposition, with values in the set of vertices of $\Omega$, and not mapping two adjacent vertices to the same vertex.
Define a map $T:\Omega\to\Omega$ by extending affinely the map $\varphi$ over each subsimplex. If all
of these affine maps are dilating (e.g. if the subsimplices are all small enough), then $T$ is a regular expanding
map which has as many sheets as there are simplices in the decomposition.

An explicit example is given by a right-angled isocele triangle, which is folded along the altitude issued from
the right-angled vertex and then rotated and dilated into the original triangle.

Just like the piecewise expanding unimodal maps above, all these examples can be considered both as IFS and expanding maps.
\end{exem}

\begin{exem}
Let $F_1,\dots F_k : \Omega\to \Omega$ be an IFS on some compact space $\Omega$,
assume the strong separation property (i.e. the $F_i(\Omega)$ are pairwise disjoints)
and up to restriction, assume $\Omega$ is the attractor (i.e. $\Omega=F_1(\Omega)\cup \dots \cup F_k(\Omega)$).
Define on $\Omega$ the map $T$ that sends $x\in F_i(\Omega)$ to $F_i^{-1}(x)$. Then $T$ is obviously
a regular expanding map.
\end{exem}

When an IFS does not have the strong separation property, we do not usually get a well-defined
expanding map. This is not a big issue since our real focus here is on
the random \emph{backward} orbits, which
are well-defined for all IFS even when they have big overlaps.

\begin{exem}
Let $\Omega=\{1,\dots,k\}^{\mathbb{N}}$ endowed with the metric
\[d_\theta(x,y)= \theta^{i(x,y)}\]
where $x=(x_j)_j, y=(y_j)$
and $i(x,y)=\min\{j\in\mathbb{N} \mid x_j\neq y_j\}$
for any fixed $\theta<1$. The shift map $\sigma:\Omega\to \Omega$ is  the transformation such that $\sigma(x_0,x_1,x_2,...)=(x_1,x_2,x_3,...)$, for any $x= (x_0,x_1,x_2,...)\in \Omega$. It is obviously
a regular expanding map with $k$ sheets and expanding
ratio $\frac1\theta$.
\end{exem}

The present framework does not cover subshifts of finite type, first because
we assume a bijection between $F(x)$ and $F(y)$ for all $x,y$ (but
it might be possible to use the multiset
approach to solve this issue), second because we ask
a bijection $(x_i,y_i)_i$ between $F(x)$ and $F(y)$ that pairs only close elements together.
It might be possible to extend the proof of the contraction property below to the case when
the \emph{average} distance between $x_i$ and $y_i$ is small, but at best at the cost of some technical
complication.

\subsection{Iterates of the transfer operator}\label{s:iterates}

We will need to consider iterates of the transfer operator, so let
us fix some notation and prove a useful estimate, to be used several
times below.

Assume that $F$ is an iterated contraction system and $A: \Omega \to \mathbb{R}$ is Lipschitz.
For each $x\in \Omega$ consider
the following multiset $\bar F^t(x)$ of \emph{admissible sequences} with respect to $F$,
of length $t+1$ and starting at $x$: $\bar F^t(x)$ contains each sequence
$s=(x_0=x,x_1,x_2,\dots, x_t)$ with $x_{n+1}\in F(x_n)$ for all $0<n<t$.
Furthermore, the sequence $(x_0=x,x_1,x_2,\dots, x_t)$ occurs with
multiplicity given by the product of the multiplicities of $x_{n+1}$ in
$F(x_n)$, for $0<n<t$. This multiset is in a natural bijection with
$F^t(x)$, but refines it by identifying the orbits
followed from $x$ to each of the elements of $F^t(x)$.

Then for each admissible sequence $s=(x,x_1,\dots, x_t)$ of length $t$,
we define
\[A^t(s) := \sum_{n=1}^t A(x_n)\]
so that, for $\varphi: \Omega \to \mathbb{R}$ continuous,
\[\mathscr{L}_A^t\varphi(x) = \sum_{s=(x,x_1x_2,\dots, x_t)\in \bar F^t(x)}   e^{A^t(s)} \varphi(x_t).\]

By definition of an ICS, for all $x$ and $y$ there is a bijection
between $\bar F^t(x)$ and $\bar F^t(y)$ such that for all admissible $s=(x,x_1,x_2,\dots, x_t)$,
the corresponding $r=(y,y_1,y_2,\dots,y_t)$ satisfies $d(x_n,y_n) \le \theta^n d(x,y)$ for all $n$.
As $A$ is Lipschitz, we hence have that
\begin{eqnarray*}
|A^t(s) - A^t(r)| & = & \left\| \sum_{n=1}^{t} A(x_n)
  - \sum_{n=1}^{t} A(y_n)\right\|
  \leq \sum_{n=1}^{t} \Lip(A) d(x_n,y_n) \\
  &\leq & \Lip(A)\sum_{n=1}^{t} \theta^{n} d(x,y)
  \leq \frac{\Lip(A)}{1-\theta} d(x,y).
\end{eqnarray*}
For all $t$, all $x,y$, and all appropriately paired
$s=(x,x_1,\dots,x_t)\in \bar F^t(x)$ and $r=(y,y_1,\dots,y_t)\in \bar F^t(y)$
we therefore have
\begin{equation}
\label{eq:multiplicative-continuity-of_A^t}
 {e^{A^t(s) - A^t(r)}} \le e^{M d(x,y)},
\end{equation}
where $M =\Lip(A)(1-\theta)^{-1}$.

\section{Normalized potentials and operators} \label{s:existence_of_eigenfunctions}

For a given Lipschitz continuous potential  $A$ and an ICS $F$, we now construct an $\mathscr{L}_{F;A}$-invariant function. Recall that the spectral radius of $\mathscr{L}_{F;A}$ acting on the space of continuous functions $C(\Omega)$ with respect to the norm $\|f\|_\infty :=  \sup_{x  \in \Omega}|f(x)|$, is
\[\rho = \lim_{n \to \infty} \left(\sup_{f \in  C(\Omega), f \neq 0} \frac{\|\mathscr{L}^n(f)\|_\infty }{ \|f\|_\infty} \right)^{\frac{1}{n}} \]

\begin{prop}\label{p:existence_of_eigenfunctions}
 Assume that $F$ is an iterated contraction system and $A: \Omega \to \mathbb{R}$ is Lipschitz. Then there exists a strictly positive, Lipschitz continuous  function $h$ such that $\mathscr{L}(h) = \rho h$.
\end{prop}

\begin{proof} We begin with the construction of $\rho$. Note that by compactness of $\Omega$, $A$ is bounded from above and below. In particular, for $n \in \mathbb{N}$,
\[  k^n e^{n \min_{x\in \Omega} A(x)} \leq \mathscr{L}^n(1)(x) \leq  k^n e^{n \max_{x\in \Omega}  A(x)}  \]
for all $x \in \Omega$. Hence, for a fixed $x_0 \in \Omega$,
\[ \tilde\rho := \limsup_{n \to \infty} (\mathscr{L}^n(1)(x_0))^{1/n}  \]
is bounded away from $0$ and $\infty$. Note that we immediately
have $\tilde\rho\le\rho$, but we will get equality later.

 Now, fix a bijection $(s^i,r^i)_{1\le i\le k^n}$ as above between
$\bar F^n(x)$ and $\bar F^n(y)$. Then
\begin{eqnarray} \label{eq:the_powers_are_Lipschitz}
\nonumber
\left|\mathscr{L}^n(1)(x)-  \mathscr{L}^n(1)(y)\right| & \leq & \sum_{i}   \left| e^{A^n(s_i)} -  e^{A^n(r_i)} \right|\\
\nonumber
 & \leq & \sum_{i}  e^{A^n(s_i)}
\left|1 -  e^{A^n(r_i) - A^n(s_i)}\right| \\
\nonumber
& \leq &\left| e^{M d(x,y)} -1 \right| \mathscr{L}^n(1)(x) \\
&\leq& \tilde{M} \mathscr{L}^n(1)(x) d(x,y),
\end{eqnarray}
with $\tilde{M} = (\exp(M \diam(\Omega)) -1)/\diam(\Omega)$.
This estimate has several important consequences. First of all, as the diameter of $\Omega$ is bounded, it follows that
\begin{equation} \label{eq:bounded_quotient} \sup\{ \mathscr{L}^n(1)(x)/\mathscr{L}^n(1)(y) : x,y \in \Omega, n \in \mathbb{N} \} < \infty,\end{equation}
which implies that $\tilde\rho$ does not depend on the choice of $x_0$;
in particular, $\tilde\rho=\rho$.

Hence, the radius of convergence of the power series
\[ \sum_{n=1}^\infty  s^n  \mathscr{L}^n(1)(x)\]
is equal to $1/\rho$ for all $x \in \Omega$. Moreover, following Denker and Urbanski (\cite{DenkerUrbanski:1991b}), there exists a sequence $(a_n)$ with $a_1=1$, $a_{n+1} \geq a_n$ and
$\frac{a_{n+1}}{a_n}\to 1$ such that
\[  \sum_{n=1}^\infty a_n s^n  \mathscr{L}^n(1)(x) \;
\left\{\begin{array}{ll} = \infty  &:\; s \geq 1/\rho \\
< \infty &:\; s < 1/\rho.
\end{array}\right.\]
Note that $(a_n)$ might be chosen independently from $x \in \Omega$ by  (\ref{eq:bounded_quotient}). For $0< s < 1/\rho$, define
\[ h_s(x) := \frac{\sum_{n=1}^\infty a_n s^n  \mathscr{L}^n(1)(x)}{\sum_{n=1}^\infty a_n s^n   \mathscr{L}^n(1)(x_0)}.\]	
It follows from (\ref{eq:bounded_quotient}) that $\|h_s\|_\infty$ is uniformly bounded, and
from (\ref{eq:the_powers_are_Lipschitz}) that $|h_s(x)-h_s(y)| \leq \tilde{M} h_s(x)  d(x,y)$. Hence, by Arz\'ela-Ascoli, there exists a sequence $(s_m)$ with $s_m \nearrow 1/\rho$ and a Lipschitz function $h$ such that $\lim_m \|h_{s_m} -h \|_\infty =0$ and $|h(x)-h(y)| \leq \tilde{M} h(x)  d(x,y)$.

We now exploit the divergence in order to show that $\mathscr{L}(h)= \rho h$. Let $\varepsilon>0$ and choose $N_\varepsilon$ such that $|a_{n-1}/a_n -1|< \varepsilon$ for all $n > N_\varepsilon$. Set $Q(s):= \sum_{n=1}^\infty a_n s^n   \mathscr{L}^n(1)(x_0)$. We then have by divergence of $Q(s)$ that
\begin{eqnarray*}
\left|\mathscr{L}(h)(x) - \rho h(x)\right|
& \leq & \lim_{m \to \infty} \frac{1}{Q(s_m)}
\left|\sum_{n=2}^\infty  (a_{n-1}s_m^{n-1} - a_n  s_m^{n} \rho) \mathscr{L}^n(1)(x) \right|\\
& = & \lim_{m \to \infty} \frac{\rho}{Q(s_m)}
\left|\sum_{n=N_\varepsilon}^\infty  \left(\frac{a_{n-1}}{\rho a_n s_m}  - 1\right) a_n s_m^n \mathscr{L}^n(1)(x) \right| \\
& \leq &\rho h(x) \sup_{n \geq N_\varepsilon} \lim_{m \to \infty}  \left|\frac{a_{n-1}}{\rho a_n s_m}  - 1\right| \leq \varepsilon \rho h(x).
\end{eqnarray*}
Hence, $\mathscr{L}(h)= \rho h$.
\end{proof}

We now employ the above proposition in order to associate a Markov chain and a corresponding Markov operator to a given ICS $F$ and a potential $A$.
\begin{defi}\label{d:Markov_chain_non_normalized}
The \emph{Markov chain associated to} the  Lipschitz potential $A$ is defined by letting
$m \cdot e^A(y)h(y)/\rho h(x)$ be the transition probability from $x$ to $y$ whenever
$y$ is an element of $F(x)$ of multiplicity $m$, where $\rho$ and $h$ are as in proposition \ref{p:existence_of_eigenfunctions}.

We denote by $\mathbb{P}^*_{F,A,h}$ (leaving again aside any subscripts that are
clear from the context) the operator on finite, signed
measures, defined by
\[\int \varphi(x) \,\dd(\mathbb{P}^*\mu)(x) = \int \sum_{y\in F(x)} e^{A(y)} \frac{h(y)}{\rho h(x)}\varphi(y) \,\dd\mu(x)\]
whenever $\varphi$ is a continuous test function. In other words,
$\mathbb{P}^*$ is the dual of the operator
defined by
\[\mathbb{P}\varphi(x) = \sum_{y\in F(x)} e^{A(y)} \frac{h(y)}{\rho h(x)} \varphi(y) = \frac{{\mathscr{L}}(h\varphi)(x)}{\rho h(x)}.\]
We refer to $\mathbb{P}$ and $\mathbb{P}^*$ as the \emph{normalized operators} with respect to $A$ and $h$. As above, since
 $\mathbb{P}(1) =1$, the dual $\mathbb{P}^*$ leaves invariant the subspace of probability measures.
\end{defi}

As a preparation for the the proofs below, we now analyze the regularity of the iterates of $\mathbb{P}$. For $s=(x, x_1, \ldots, x_t) \in \bar F_t$ as defined above, set
\[ A_h^t(s) = A^t(s) + \log h(x_t) - \log h(x)  - n \log \rho.\]
As it easily can be seen, we then have that
\[\mathbb{P}^t \varphi(x)  = \sum_{s=(x, \ldots x_t), s \in \bar F^t(x)} e^{A^t_h(s)}  \varphi(x_t).\]
 Furthermore, for $r,s \in \bar F_t$ appropriately paired with $r=(y,y_1,\ldots, y_t)$, it follows that
\begin{eqnarray} \label{eq:multiplicative-continuity-of_A^t-normalized-operator}
\nonumber e^{A_h^t(s) - A_h^t(r)} & = e^{A^t(s) - A^t(r)} \frac{h(x_t)}{h(y_t)} \frac{h(y)}{h(x)} \\
\nonumber & \leq e^{M d(x,y)} \left(1 + \tilde{M} \theta^t d(x,y) \right) \left( 1 + \tilde{M}  d(x,y)  \right)   \\
& \leq  e^{(M + 2 \tilde{M} ) d(x,y)} =   e^{M' d(x,y)},
\end{eqnarray}
where $M' = M + 2 \tilde{M} $.

\section{Optimal transport and Wasserstein metric}

We will need to use coupling in order to derive our main results. In order to do so, let us give a simple but useful technical
result.

\begin{prop}\label{p:partition}
Assume that there are sets $A_1,\dots ,A_n$ such
that the probability measures $\mu$ and $\nu$ are concentrated on the union of the $A_i$. Let
$c=\max_i \diam(A_i)$, $C=\diam(\cup A_i)$ and $m=\sum_i \min(\mu(A_i),\nu(A_i))$.
Then
\[ W_1(\mu,\nu) \le m c+(1-m)C.\]
\end{prop}

\begin{proof}
We let $\pi$ be a coupling of $\mu$ and $\nu$ that moves a mass
at most $m$ between different $A_i$'s, i.e. such that
\[\pi(\{(x,y) | \exists i \mbox{ such that both } x,y\in A_i\}) \ge m.\]
Once this transport plan is constructed, we compute
\begin{eqnarray*}
\int_{\Omega\times \Omega} d(x,y) \,\dd\pi(x,y) &=&
  \int_{\cup_i A_i\times A_i} d(x,y) \,\dd\pi(x,y) +
  \int_{\Omega\setminus\cup_i A_i\times A_i} d(x,y) \,\dd\pi(x,y) \\
  &\le& m c + (1-m)C.
\end{eqnarray*}
To construct $\pi$, we first note that it is possible to
decompose $\mu$ into
\[\mu=\sum_i(\mu_i^{\mathrm{in}}+\mu_i^{\mathrm{out}})\]
 where the
$\mu_i^{\mathrm{in}/\mathrm{out}}$ are concentrated on $A_i$ and
$\mu_i^{\mathrm{in}}(A_i)=\min(\mu(A_i),\nu(A_i))$ (and similarly for $\nu$).
Then we set
\[\pi = \sum_i \mu_i^{\mathrm{in}} \otimes \nu_i^{\mathrm{in}} +
  \big(\sum_i \mu_i^{\mathrm{out}}\big)\otimes \big(\sum_i \nu_i^{\mathrm{out}}\big).\]
\end{proof}

The following proposition is also more or less folklore and very useful; it appears
for example in a proof in \cite{HairerMattingly:2008}.
\begin{prop}\label{p:linear}
Let $P$ be a linear operator on the set of measures
on $\Omega$ (assumed to be compact for simplification), such that
$P$ is continuous in the weak-$\ast$ topology and maps probability
measures to probability measures.

If for some $C>0$ and all $x,y$ in some dense subset of $\Omega$ we have
\[W_1(P(\delta_x),P(\delta_y))\le C d(x,y)\]
then for all $\mu,\nu\in\mathscr{P}(\Omega)$ we also have
\[W_1(P(\mu),P(\nu)) \le C W_1(\mu,\nu).\]
\end{prop}

\begin{proof}
Let us give a slight variation of the Hairer-Mattingly
proof, using density of finitely supported measures: we only have to
prove $W_1(P(\mu),P(\nu)) \le C W_1(\mu,\nu)$ when
$\mu=\sum_{i\in I} a_i \delta_{x_i}$ and $\nu=\sum_{j\in J} b_j \delta_{y_j}$
and $x_i,y_j$ are in the dense subset of $\Omega$ we are given.
Let
\[\tilde\pi =  \sum_{i\in I\,,\,j\in J} c_{i,j} \delta_{(x_i,y_j)}\]
be an optimal transport plan from $\mu$ to $\nu$, and
for each $(i,j)$, let $\pi_{i,j}$
be an optimal transport plan from $P(\delta_{x_i})$ to $P(\delta_{y_j})$.

Define $\pi=\sum_{i,j} c_{i,j}\pi_{i,j}$; it transports $P(\mu)$
to $P(\nu)$ and we have
\begin{eqnarray*}
\int_{\Omega\times \Omega} d(x,y) \,\dd\pi(x,y) &=& \sum_{i,j} c_{i,j} \int d(x,y) \,\dd\pi_{i,j}(x,y) \\
   &=& \sum_{i,j} c_{i,j} W_1\big(P(\delta_{x_i}),P(\delta_{y_j})\big) \\
   &\le & C \sum_{i,j} c_{i,j} d(x_i,y_j) \\
   &= & C W_1(\mu,\nu)
\end{eqnarray*}
proving the claim.
\end{proof}

\section{Proof of the main result and first applications}

We are now in position to prove the main theorem. Throughout this section, assume that $F$ is a $k$-ICS with contraction ratio $\theta$, $A$ is a Lipschitz potential on the attractor $\Omega$ of $F$, and
$\mathbb{P}$ and $\mathbb{P}^*$ are defined as in definition \ref{d:Markov_chain_non_normalized}. For the reader's convenience, the statement of Theorem \ref{ti:contraction} is repeated.

\begin{theo}[Contraction property]
\label{t:contraction}
The normalized operator $\mathbb{P}^*$ is exponentially contracting on probability measures:
There exist constants $C=C(\Lip(A),\theta,\diam\Omega)$ and $\lambda=\lambda(\Lip(A),\theta,\diam\Omega)<1$ such that
for all $n\in\mathbb{N}$ and all $\mu,\nu\in\proba(\Omega)$ we have
\[W_1((\mathbb{P}^*)^n\mu,(\mathbb{P}^*)^n\nu) \le C\lambda^n W_1(\mu,\nu).\]
\end{theo}

\begin{proof} We use three reductions of the problem. First, it is sufficient to prove Theorem
\ref{t:contraction} for some iterate $(\mathbb{P}^*)^t$ of the dual of the normalized operator
(using the continuity of the operator and the flexibility given by the constant $C$).
Second, it is sufficient to prove it when $\Omega$ is endowed with any metric
$d'$ which is Lipschitz-equivalent to $d$
(again using the constant $C$ to absorb
the ratio between the two metrics); an important point is that we can choose the metric
$d'$ depending on $A$. Last, thanks to Proposition \ref{p:linear},
we only need to prove it when $\mu$ and $\nu$ are Dirac measures.

So, it is sufficient to find $t\in\mathbb{N}$, a metric $d'$ equivalent to $d$
and a number $\lambda'\in(0,1)$ such that for all $x,y\in\Omega$ we have
\[W'_1((\mathbb{P}^*)^t\delta_x,(\mathbb{P}^*)^t \delta_y) \le \lambda' d'(x,y)\]
where $W'_1$ is the Wasserstein metric associated to the distance $d'$.

The principal idea is to apply Proposition \ref{p:partition}; let us define
\[d'(x,y) =
\left\{ \begin{array}{ll}
  \theta^{-N} d(x,y) &\mbox{if }d(x,y)\le \theta^N \cdot \diam\Omega\\
  \diam\Omega &\mbox{otherwise}
\end{array} \right.\]
for some $N$ to be specified later. This metric will make Proposition \ref{p:partition}
more effective because it localizes the Wasserstein metric to some small scale
(all displacements are now equivalent as soon as they are somewhat big).

Now fix a positive integer $t$. Moreover, for $x,y \in \Omega$, fix a bijection $(s^i,r^i)_{1\le i\le k^t}$ between
$\bar F^t(x)$ and $\bar F^t(y)$ as in Section  \ref{s:iterates} and apply a slight variant of Proposition \ref{p:partition}:
let $\pi$ refer to a transport plan from
\[(\mathbb{P}^*)^t\delta_x = \sum_{i} e^{A_h^t(s^i)}\delta_{x_t}\]
to
\[(\mathbb{P}^*)^t\delta_y = \sum_{i} e^{A_h^t(r^i)}\delta_{y_t}\]
that moves a mass at least (cf. estimate (\ref{eq:multiplicative-continuity-of_A^t-normalized-operator}))
\begin{eqnarray*}
m(x,y) &:= &\sum_i \min(e^{A_h^t(s^i)},e^{A_h^t(r^i)}) \ge \sum_i e^{A_h^t(s^i)} e^{-M'd(x,y)}\\
  & = & e^{-M'd(x,y)}
\end{eqnarray*}
by a distance at most $d'(x^i_t,y^i_t)\le \theta^{t-N}d'(x,y)$
and moves the rest of the mass by a distance at most $\diam\Omega$.
We get
\begin{eqnarray*}
W'_1\Big((\mathbb{P}^*)^t\delta_x,(\mathbb{P}^*)^t\delta_y\Big)
  &\le & e^{-M'd(x,y)}\theta^{t-N}d'(x,y) + (1-e^{-M'd(x,y)})\diam\Omega \\
  &\le & \theta^{t-N}d'(x,y)+(1-e^{-M' d(x,y)})\diam\Omega,
\end{eqnarray*}
which is at most
\[
\left\{\begin{array}{l l} (\theta^{t-N} + M'\cdot\diam\Omega\cdot\theta^N) d'(x,y) &\mbox{when } d'(x,y)<\diam \Omega \\
      \big(\theta^{t-N}+1-e^{-M'\diam\Omega}\big)\cdot \diam\Omega &\mbox{when }d'(x,y)=\diam\Omega
      \end{array}
      \right.
 \]
First note that the expressions above only depend on
the parameters $\theta$, $\diam\Omega$, $\Lip(A)$. Now, taking $N$ large enough
and then $t$ large enough ensures that
the right-hand-side is at most $\lambda' d'(x,y)$ for some uniform $\lambda'<1$.
\end{proof}

If $A$ already is a normalized potential, the constants in the above theorem can be determined rather explicitly. Namely, it
is not difficult to see that one can take for example
\[C= \theta^{-N}\Big(\theta+\frac{M}{1-\theta}\diam\Omega\Big)^{2t}\]
(recall that $M =\Lip(A)(1-\theta)^{-1}$) and
\[\lambda = \Big(1-\frac12e^{-\frac{M}{1-\theta}\diam\Omega}\Big)^{\frac1t}\]
where $N$ is the solution to
\[\theta^N\frac{M}{1-\theta}\diam\Omega = 1-e^{-\frac{M}{1-\theta}\diam\Omega}\]
and $t$ is such that
\[\theta^t\le \theta^{2N}\frac{M}{2(1-\theta)}\diam\Omega.\]

Note that $\lambda$ depends on $t$ and  that $C$ depends on $N$ and also on $t$. Playing with $N$ and $t$ we 
can improve $\lambda$. These two values $C$ and $\lambda$ are important in the next result.

\subsection{Proof of the existence of a spectral gap}

Through duality, it is now easy to prove Corollary \ref{ci:spectralgap}
and deduce uniqueness of $h$.
\begin{coro}[Spectral gap]\label{c:spectralgap}
Let $\mu$ be the fixed point of
$\mathbb{P}^*$ in $\proba(\Omega)$ (i.e. the invariant Gibbs measure associated to
$F$ and $A$); for each Lipschitz function
$\zeta:\Omega\to\mathbb{R}$ such that $\int \zeta \,\dd\mu=0$, we have
\[\Lipnorm{\mathbb{P}^n\zeta} \le C_2(\zeta)\lambda^n\]
where $C_2(\zeta)=(1+\diam\Omega)C\Lip(\zeta)$ and
$C,\lambda$ are the constants given by Theorem \ref{t:contraction}.
In particular, the function $h$ in proposition \ref{p:existence_of_eigenfunctions} is unique up to multiplication by constants.
\end{coro}

\begin{proof}
We first control the uniform norm of $\mathbb{P}^n\zeta$ (this is the part
where we need $\zeta$ to have vanishing $\mu$-average): for all $x\in\Omega$ we have
\begin{eqnarray*}
\left\| \mathbb{P}^n\zeta(x)\right\|
  &= &\left\| \int \mathbb{P}^n\zeta(y) \,\dd\delta_x(y) -\int \zeta \,\dd\mu \right\| \\
  &=& \left\| \int \zeta(y) \,\dd\big(\mathbb{P}^{*n}\delta_x\big)(y)
      -\int \zeta \,\dd\mu \right\|\\
  &\le & \Lip(\zeta) W_1(\mathbb{P}^{*n} \delta_x,\mu)
  \le \Lip(\zeta) \cdot C\lambda^n W_1(\delta_x,\mu)\\
  &\le & C\diam\Omega\cdot\Lip(\zeta)\cdot\lambda^n.
\end{eqnarray*}

Next we control with the same kind of trick the Lipschitz constant of
$\mathbb{P}^n\zeta$ (this part holds whatever the integral of $\zeta$):
for all $x,y$ we have
\begin{eqnarray*}
\left\| \mathbb{P}^n\zeta(x)-\mathbb{P}^n\zeta(y)\right\|
  &=& \left\| \int \mathbb{P}^n\zeta \,\dd\delta_x
     -\int \mathbb{P}^n\zeta \,\dd\delta_y \right\|\\
  &=& \left\| \int \zeta \,\dd\big(\mathbb{P}^{*n}\delta_x\big)
      -\int \zeta \,\dd\big(\mathbb{P}^{*n}\delta_y\big) \right\| \\
  &\le& \Lip(\zeta) W_1(\mathbb{P}^{*n} \delta_x,\mathbb{P}^{*n} \delta_y)
  \le \Lip(\zeta) \cdot C\lambda^n d(x,y).
\end{eqnarray*}
This also implies that $\mathbb{P}f = f$ if and only if $f$ is a constant function. Hence, $\mathscr{L}(f)=f$ if and only if $f$ is a multiple of $h$ given by  proposition \ref{p:existence_of_eigenfunctions}.
\end{proof}

Observe that this result for example implies that an expression like
\[\sum_{n=0}^\infty \mathbb{P}^n \zeta\]
is a well-defined Lipschitz function whenever $\int\zeta \,h \dd\mu=0$.
This expression moreover defines a bounded inverse to the operator
$I-\mathbb{P}_{F,A}$ restricted to $0$-average functions.

When $F$ is induced by a map $T$, it is also classical to deduce
an exponential decay of correlations from the spectral gap; however,
in our general setting and given the way $\mathbb{P}$ is defined,
we would need to extend to regular expanding maps the classical relation
\[\int f\circ T \cdot g \,\dd\mu = \int f \cdot \mathbb{P}(g) \,\dd\mu\]
(for all $f\in L^1(\mu)$ and $g$ continuous). This is certainly doable,
but needs to carefully handle measurable selections; to keep the
present article relatively short, we prefer to postpone these details
to a further study of ICS and regular expanding maps.

\section{Stability of the Gibbs map} \label{s:Stability of the Gibbs map}
Unless otherwise specified, we assume throughout this section that the
potentials are already normalized (i.e. $\mathscr{L}=\mathbb{P}$)
in order to be able to give accessible proofs which reveal the interplay between
coupling techniques and thermodynamic formalism. Moreover, this also allows to give relatively
explicit controls on the associated constants.

\subsection{General results}

In order to prove that the map which sends an ICS $F$ and a normalized potential $A$ to the Gibbs measure $\mu_{F,A}$
is locally Lipschitz, we first need to prove the stability of the dual transfer operator.

The uniform norm $\left\|\cdot\right\|_\infty$ is defined as usual for potentials,
and a similar distance is defined for ICS with the same number of terms defined on a common
metric space $X$ by:
\[d_\infty(F_1,F_2)=\sup_{x\in X}\, \inf_{(y_1^j,y_2^j)_j}\,\sup_j \, d(y_1^j,y_2^j)\]
where the infimum is taken over all bijections between the multisets $F_1(x)$ and $F_2(x)$.
In other words, $d_\infty(F_1,F_2)\le D$ exactly when for all $x$, it is possible
to pair the elements of $F_1(x)$ and $F_2(x)$ such that no two paired elements are
more than $D$ apart.

\begin{prop}
Let $F_1,F_2$ be two ICS with $k$ terms defined on the
same compact metric space $X$. Let $A_1,A_2$ be
potentials defined on $X$ which are assumed to be normalized with respect to $F_1$
and $F_2$ respectively. Let $\mathscr{L}_i=\mathscr{L}_{F_i,A_i}$ be the transfer operator
defined by $(F_i,A_i)$ on the set of continuous functions from $X$ to $\mathbb{R}$.
Then
for any probability measure $\mu$ on $X$, we have
\[W_1(\mathscr{L}_1^*\mu , \mathscr{L}_2^*\mu) \le \diam X \cdot \left\| A_1-A_2\right\|_\infty
  +(\Lip(A_2)\diam X+1) d_\infty(F_1,F_2).\]
\end{prop}
This inequality is not optimal from the proof below, but is good enough for small variations
and easy to state. Note that by symmetry, $\Lip(A_2)$ can be replaced by $\Lip(A_1)$,
the point being that we only need to control one of the Lipschitz constants.

\begin{proof}
Reasoning as in the proof of Proposition \ref{p:linear}, we see that it is sufficient to
prove this inequality when $\mu=\delta_x$ is a Dirac mass. In this case, we have
\[\mathscr{L}_i^*\delta_x = \sum_{j=1}^k e^{A_i(y_i^j)} \delta_{y_i^j};\]
where $y_i^1,\dots y_i^k$ are the elements of $F_i(x)$, numbered such that
$d(y_1^j,y_2^j)\le d_\infty(F_1,F_2)$ for all $j$. There is a transport plan
between these two measures that moves as much mass as possible from each of the
$y_1^j$ to $y_2^j$. This plan moves an amount of mass
\[m(x):=\sum_j \min(e^{A_1(y_1^j)},e^{A_2(y_2^j)})\]
by a distance at most $d_\infty(F_1,F_2)$, and the rest of the mass
is moved by at most $\diam X$.

We have for all $j$:
\begin{eqnarray*}
A_2(y_2^j) &\ge & A_2(y_1^j) -\Lip(A_2)d(y_1^j,y_2^j) \\
  & \ge & A_1(y_1^j) - \left\| A_1-A_2\right\|_\infty -\Lip(A_2) d_\infty(F_1,F_2)
\end{eqnarray*}
so that
\[e^{A_2(y_2^j)} \ge e^{A_1(y_1^j)} e^{- \left\| A_1-A_2\right\|_\infty -\Lip(A_2) d_\infty(F_1,F_2)},\]
from which it comes (using the normalization $\sum e^{A_1(y_1^j)}=1$)that
\[1-m(x)\le \left\| A_1-A_2\right\|_\infty +\Lip(A_2) d_\infty(F_1,F_2).\]
We get that the plan under consideration has cost less than
\[m(x)d_\infty(F_1,F_2)+\diam X\left(\left\| A_1-A_2\right\|_\infty +\Lip(A_2) d_\infty(F_1,F_2)\right)\]
and bounding $m(x)$ by $1$ yields the claimed inequality.
\end{proof}

Combining this estimate with the contraction property, we obtain that the Gibbs measure
depends on the ICS and the potential in a locally Lipschitz way.

\begin{coro}\label{c:Gibbs-map}
Let $F_1,F_2$ be two ICS with $k$ terms defined on the
same compact metric space $X$.\footnote{with possibly different attractors $\Omega_1,\Omega_2$.} Let $A_1,A_2$ be
potentials defined on $X$ which are assumed to be normalized with respect to $F_1$
and $F_2$ respectively. Let $\mu_i$ be the Gibbs measure associated with $(F_i,A_i)$,
i.e. the unique probability measure invariant under $\mathscr{L}_i^*=\mathbb{P}_i^\ast$.

If $F_2$ has contraction ratio $\theta$ then we have
$$
W_1(\mu_1,\mu_2) \le \frac{C}{1-\lambda} \big( \diam X \cdot \left\| A_1-A_2\right\|_\infty
  +(\Lip(A_2)\diam X+1) d_\infty(F_1,F_2)\big)$$
where $C,\lambda$ are the constants given by Theorem \ref{t:contraction} in terms of $\diam X$,
$\theta$ and $\Lip(A_2)$.
\end{coro}
Note that if we vary both pairs $(F_i,A_i)$, we only get a \emph{locally} Lipschitz control,
as $C$ and $\lambda$ both get poor when $\Lip(A_2)$ goes to infinity, or $\theta$ goes to $1$.
But if we fix one of them, $(F_2,A_2)$ say, then we get a globally uniform control of the
distance between the Gibbs measures.

\begin{proof}
Consider
\[u_n:=\sup_{\mu\in\proba(X)} W_1(\mathscr{L}_1^{*n}\mu,\mathscr{L}_2^{*n}\mu);\]
from the previous proposition we know that
\[u_1\le  \diam X \cdot \left\| A_1-A_2\right\|_\infty+(\Lip(A_2)\diam X+1) d_\infty(F_1,F_2).\]
Given any probability measure $\mu$ on $X$, we have
\begin{eqnarray*}
& & W_1(\mathscr{L}_1^{*(n+1)}\mu,\mathscr{L}_2^{*(n+1)}\mu) \\
  &\le &W_1\big(\mathscr{L}_1^{*n}(\mathscr{L}_1^{*} \mu),\mathscr{L}_2^{*n}(\mathscr{L}_1^{*} \mu)\big)
   +W_1\big(\mathscr{L}_2^{*n}(\mathscr{L}_1^{*} \mu),\mathscr{L}_2^{*n}(\mathscr{L}_2^{*} \mu) \big) \\
 & \le & u_n+C\lambda^n W_1(\mathscr{L}_1^{*}\mu,\mathscr{L}_2^{*}\mu)\\
 &\le & u_n+C\lambda^n u_1.
\end{eqnarray*}
Then by induction on $n$ we get
\[u_n \le (C\lambda^{n-1}+\dots+C\lambda^2+C\lambda+1) u_1 \le \frac{C}{1-\lambda} u_1.\]

For any fixed probability $\mu$,
when $n$ goes to $\infty$, we have $\mathscr{L}_i^{*n}\mu\to\mu_i$ so that
we get
\[W_1(\mu_1,\mu_2) \le \liminf u_n \le \frac{C}{1-\lambda} u_1\]
as desired.
\end{proof}

We can now easily deduce the results announced in the introduction starting with the following.
\begin{proof}[Proof of Corollary \ref{ci:Gibbs-map}]
We simply apply Corollary \ref{c:Gibbs-map} to $F_1=F_2=F$
and $A,B$, getting:
\[W_1(\mu_A,\mu_B) \le \frac{C}{1-\lambda}\diam \Omega \cdot \left\| A-B\right\|_\infty.\]
The consequence in term of test functions follows by duality.
%
%
\end{proof}

\subsection{Application to expanding maps}

Let us now see how the above can be used to prove Corollaries
\ref{ci:metric-entropy} and \ref{ci:max-entropy} above for expanding maps with respect to normalized potentials.

\begin{proof}[Proof of Corollary \ref{ci:metric-entropy}]
Since $A$ and $B$ are normalized, the spectral radii of the $\mathscr{L}_A$ and $\mathscr{L}_B$ are equal to $1$. Furthermore, $\mu_A$ and $\mu_B$ are equilibrium states (see, e.g., \cite{Walters:1978}). Hence,
$h(\mu_A)=-\int A \,\dd\mu_A$ and $h(\mu_B)=-\int B \,\dd\mu_B$. Using the previous
inequality we get:
\begin{eqnarray*}
\left\| h(\mu_A) - h(\mu_B) \right\| &\le & \left\| \int A \,\dd\mu_A - \int A \,\dd\mu_B \right\|
    + \int \left\| A-B\right\| \,\dd\mu_B \\
  &\le &\Lip(A) W_1(\mu_A,\mu_B) + \left\| A-B\right\|_\infty \\
  &\le & \Big(\frac{C\Lip(A)}{1-\lambda}\diam \Omega +1\Big) \left\| A-B\right\|_\infty.
\end{eqnarray*}
\end{proof}

To prove Corollary \ref{ci:max-entropy}, we mainly have
to show how the ICS $F$ depends on the given expanding map $T$. This is the part
where we restrict to $C^1$ expanding maps on manifolds.
\begin{lemm}
Let $T_1$, $T_2$ be $C^1$ expanding map on the same manifold $\Omega$ and assume
that $\left\| T_1-T_2\right\|_\infty \le \frac14 \sys(\Omega)$. Then the ICS
$F_i:x\mapsto T_i^{-1}(x)$ satisfy
\[d_\infty(F_1,F_2) \le 2d_\infty(T_1,T_2).\]
\end{lemm}

\begin{proof}
First, recall that both $T_1$ and $T_2$ are self-covering maps of $\Omega$.

Let $x\in\Omega$ be any point, and let
\[\{x_1,\dots, x_k\}:=T_1^{-1}(x)=F_1(x).\]
For all $j\in\{1,\dots,k\}$, let $\gamma_j$ be a shortest geodesic from
$x$ to $T_2(x_j)$ and denote by $\gamma_j^{-1}$ the same curve parametrized
in the other direction; note that these curves have length at most $d_\infty(T_1,T_2)$.
We construct a curve $\tilde\gamma_j$ in $\Omega$ as follows.

First, $\tilde\gamma_j^1$ is the lift of $\gamma_j$
with respect to the covering map $T_1$ that starts at $x_j$. Its endpoint is
mapped by $T_1$ to $T_2(x_1)$. Second, $\tilde\gamma_j^1$ is the lift
of $\gamma_j^{-1}$ with respect to the covering map $T_2$ that starts at the endpoint
of $\tilde\gamma_j^1$; its endpoint is denoted by $y_j$ and we have
$T_2(y_j) = x$. Then $\tilde\gamma_j$ is the concatenation of
$\tilde\gamma_j^1$ and $\tilde\gamma_j^2$.

By construction, $\tilde\gamma_j$ links $x_j\in F_1(x)$ to $y_j\in F_2(x)$
and, since the $T_i$ are expanding, has length at most
$2 d_\infty(T_1,T_2)$. Our assumption on the distance between the $T_i$
ensures that the $y_j$ are pairwise distinct, so that
$F_2(x)=\{y_1,\dots,y_k\}$; the conclusion then follows from the definition
of the uniform distance between ICS.
\end{proof}

\begin{proof}[Proof of Corollary \ref{ci:max-entropy}]
It is well-known (see, e.g., \cite{Walters:1978}) that the maximal entropy measure
of $T_i$ is the Gibbs measure associated to the constant potential $A=-\log k$
where $k$ is the number of sheets of $T_i$. We only have to apply
Corollary \ref{c:Gibbs-map} with $A_1=A_2=A$ (so that in particular $\Lip(A)=0$),
using the previous Lemma
to control $d_\infty(F_1,F_2)$, to get the desired conclusion.
\end{proof}

\section*{References}
\bibliographystyle{alpha}
\bibliography{biblio}

\end{document}